\documentclass[sts,
	reqno,
]{imsart}

\usepackage{amsthm, amssymb, amsmath}
\usepackage{amsmath}               
  {
      \theoremstyle{plain}
      \newtheorem{assumption}{Assumption}
  }
\usepackage{empheq}
\usepackage{euscript}
\usepackage{mathtools}
\usepackage{mdframed}
\usepackage{tcolorbox}
\usepackage{color}
\usepackage{sectsty}
\usepackage{tikz}
\usetikzlibrary{arrows}
\usepackage[utf8]{inputenc}
\usepackage[shortlabels]{enumitem}
\usepackage{url}
\usepackage{cases} 
\definecolor{labelkey}{rgb}{0.0, 0.8, 0.3}
\usepackage[top=2.5cm,bottom=2.5cm,left=2.5cm,right=2.5cm,includeheadfoot,asymmetric]{geometry}

\usepackage{macros}
\renewcommand{\le}{\leqslant}
\renewcommand{\ge}{\geqslant}
\renewcommand{\phi}{\varphi}

\newcommand{\p}{\mathbb{P}}
\newcommand{\KL}{\mathsf{KL}}
\renewcommand{\abstractname}{Abstract}
\newcommand{\email}[1]{(\href{mailto:#1}{\rm{#1}})}

\begin{document}

\begin{frontmatter}

	\title{On the sample complexity of entropic optimal transport}
	\runtitle{Sample complexity of entropic OT}
 
	\author{ 
		\fnms{Philippe} \snm{Rigollet}\ead[label=rigollet]{rigollet@math.mit.edu}
			~and~
		\fnms{Austin J.} \snm{Stromme}\ead[label=ajs]{astromme@mit.edu}
	}
	\address{{Philippe Rigollet}\\
		{Department of Mathematics} \\
		{Massachusetts Institute of Technology}\\
		{77 Massachusetts Avenue,}\\
		{Cambridge, MA 02139-4307, USA}\\
		{\email{rigollet@math.mit.edu}}
	}
	\address{{Austin J. Stromme}\\
		{Department of EECS} \\
		{Massachusetts Institute of Technology}\\
		{77 Massachusetts Avenue,}\\
		{Cambridge, MA 02139-4307, USA}\\
		{\email{astromme@mit.edu}}
	}
	\affiliation{Massachusetts Institute of Technology}

\runauthor{Rigollet and Stromme}

\begin{abstract}
We study the sample complexity of entropic optimal
transport
in high dimensions using computationally
efficient plug-in estimators. We significantly advance the state
of the art by establishing dimension-free, parametric rates
for estimating various quantities of interest, including  the entropic regression function which is a natural analog to the optimal transport map. As an application, we propose a practical model
for transfer learning based
on entropic optimal transport and establish parametric rates of convergence for  nonparametric  regression and classification.
\end{abstract}

\end{frontmatter}

\section{Introduction}

Thanks to remarkable computational advances~\cite{PeyCut19}, optimal transport (OT) has recently emerged as an effective tool to tackle a wide range of statistical problems that were out of reach of previous methods.  Given two measures $\mu$ and $\nu$ on $\R^d$ and a cost function $c:\R^d \times \R^d \to [0,\infty)$,  the OT problem~\cite{Vil03,Vil09} is the infinite dimensional linear optimization problem given by
\begin{equation}
    \label{eq:OTmain}
    \inf_{\pi \in \Pi(\mu,\nu)} 
\int c(x,y)\ud \pi (x,y)\,,
\end{equation}
where the infimum is taken over the set $\Pi(\mu, \nu)$ of couplings  between $\mu$ and $\nu$. Recall that $\pi \in \Pi(\mu, \nu)$ is a valid coupling between $\mu$ and $\nu$ if $\pi$ is a probability measure on $\R^d\times \R^d,$ such that for any measurable $A \subset \R^d$, it holds that $\pi(A \times \R^d)=\mu(A)$ and $\pi(\R^d \times A)=\nu(A)$. Throughout this paper, we assume that $\mu$ and $\nu$ have bounded support. Moreover, for concreteness we set $c(x,y)=\|x-y\|^2$  as in the majority of applications of OT, but  our results readily extend to any cost $c$ that is uniformly bounded on the support of $\mu$ and $\nu$. In that case,~\eqref{eq:OTmain} admits a unique minimizer $\pi_0$, called the \emph{OT coupling}. Therefore, OT provides a principled way of selecting a non-trivial coupling between probability measures. 
Furthermore, Brenier's theorem  states that under mild regularity conditions on $\mu$, the OT coupling $\pi_0$ is supported on the graph of a deterministic map $T$ called the \emph{Brenier map}. In other words, $(X,Y) \sim \pi_0$ if and only if $X \sim \mu$ and $Y=T(X)\sim \nu$. 

The OT coupling has the following dynamical interpretation in terms of energy minimization. Consider the position $X_t$ at time $t$ of a particle in $\R^d$ evolving according to a time-varying velocity field $v_t: \R^d \to \R^d$ so that $\ud X_t=v_t(X_t) \ud t$. It turns out that the velocity field $\{v_t\}$ that transports a population of particles from $X_0\sim \mu$ to $X_1 \sim \nu$ while minimizing the (kinetic) energy functional  $\int \E\|v_t(X_t)\|^2 \ud t$ is given simply in terms of the Brenier map: it is the velocity field that transports $x$ to $T(x)$ at constant speed along a straight line. Moreover, the minimum energy value is given by the squared Wasserstein distance which is also the value of the minimum in~\eqref{eq:OTmain}. This is the Benamou-Brenier formula~\cite[Prop. 5.30]{San15}. This argument extends to the case where the Brenier map may not exist and it shows that the OT coupling minimizes the energy needed to evolve a population of particles from an initial to a final distribution.
This observation has fueled a conceptual shift from the traditional statistical toolbox in cases where the energy minimization perspective is justified. Indeed, the estimation of couplings between datasets has led to spectacular developments on fundamental questions in many areas including statistics~\cite{RigWee18,rigollet2019uncoupled}, economics~\cite{TorGunRig21}, computer graphics~\cite{SolGoePey15}, computational biology~\cite{schiebinger2019optimal, LavZhaKim21}, and machine learning~\cite{ot_for_domain_adaptation}. 

A central application of optimal transport is \emph{transfer learning}, where the goal is to transfer information from one dataset to another using the Brenier map. For example, in~\cite{ForHutNit19}, an estimated transport map is constructed from a labeled source dataset to an unlabeled target dataset in order to perform classification on the target dataset despite an absence of labels. This question has received a surge of interest in the context of image classification under the name \emph{domain adaptation}~\cite{ot_for_domain_adaptation}. Here, the goal is to automatically adapt image classification under shifting image conditions such as lighting. For this class of problems, optimal transport provides a natural candidate to transfer the points using the Brenier map due to its minimum energy property.

Unfortunately, a line of recent
work has provided strong evidence that the OT coupling
suffers from a statistical curse of dimensionality.
Indeed, without further assumptions, the minimax rate for estimating the  OT cost
is at least $n^{-1/d}$~\cite{niles2019estimation}, and a similar rate is conjectured to 
hold for the problem of estimating the Brenier map $T$~\cite{HutRig21}.
While recent theoretical effort has been devoted to showing
that this inefficiency can be alleviated  by making structural assumptions---chiefly smoothness---on the transport map, finding computationally efficient and smoothness-adaptive estimators
is a challenging and ongoing research topic~\cite{ForHutNit19,HutRig21,ManBalNil21,manole2021sharp,pooladian2021entropic,VacMuzRud21,muzellec2021near,deb2021rates}.

In this work, we study an alternative to the OT coupling that we call the
\emph{Schr\"odinger coupling}. It arises as the solution to the entropically regularized OT problem given by,
for $\eta > 0$,
\begin{equation}
    \label{eq:SCHmain}
 \inf_{\pi \in \Pi(\mu,\nu)} \Big\{ 
\int \|x-y\|^2\ud \pi (x,y) + \frac{1}{\eta} \KL(\pi \| \mu \otimes \nu)\Big\}\,,
\end{equation}
where $\KL$ denotes the Kullback-Leibler divergence.
This regularized problem dates back to
early work of Schr\"odinger~\cite{Sch31,schrodinger1932theorie},
and recently has largely eclipsed the OT coupling
in applications because it offers significant computational advantages. Indeed, it can be computed extremely quickly using the Sinkhorn algorithm~\cite{sinkhorn1964relationship, Cut13, altschuler2017near, PeyCut19}. 
Like the OT coupling, the Schr\"odinger coupling also arises
from a minimum energy paradigm. However, in the
Schr\"odinger coupling, the particles evolve according to a {\it stochastic} differential equation $\ud X_t=v_t(X_t) \ud t + (2\eta)^{-1/2}\ud W_t$, where $\{W_t\}$ is a standard Brownian motion over $\R^d$. In that case, it can be shown using the Girsanov formula that the velocity field that minimizes the energy functional $\int \E\|v_t(X_t)\|^2 \ud t$ subject to $X_0 \sim \mu$ and $X_1 \sim \nu$ induces the Schr\"odinger coupling; see, e.g., \cite{leonard2013survey}.

Most previous works have studied the Schr\"odinger coupling
in the asymptotic regime $\eta \to \infty$ as a surrogate for the original OT problem~\cite{erbar2015large,pal2019difference,nutz2021entropic,bernton2021entropic,altschuler2022asymptotics, delalande2022nearly}.
In light of the minimum energy interpretation above, we argue that the Schr\"odinger coupling is a quantity of interest on its own. As a result, we treat instead
 $\eta$ as a fixed parameter.
Recall that we assume that the probability measures $\mu$ and $\nu$
have bounded support. In this case,
the Schr\"odinger coupling exists and is unique, and we denote it by
\begin{equation}\label{eqn:pi_star_defn}
\pi_\star := \argmin_{\pi \in \Pi(\mu, \nu)} \Big\{ \int \|x - y\|^2 \ud \pi(x, y)
+ \frac{1}{\eta} \KL(\pi \| \mu \otimes \nu) \Big\}\, ,
\end{equation}
This paper focuses on the estimation of the Schr\"odinger coupling $\pi_\star$ and quantities that are derived from it.

We work in a standard statistical setting. Assume
$X_1, \ldots, X_n \sim \mu$ i.i.d., 
and $Y_1, \ldots, Y_n \sim \nu$ i.i.d.
We denote the sample from $\mu$ by $\X := (X_1, \ldots, X_n)$
and the sample from $\nu$ by $\Y := (Y_1, \ldots, Y_n)$. We assume further that the samples $\X$ and $\Y$ are mutually independent. 
The associated empirical measures are denoted $\mu_n$ and $\nu_n$ respectively. Recall that they are given by
$$
\mu_n := \frac{1}{n} \sum_{i = 1}^n \delta_{X_i} \quad \quad 
\nu_n:=\frac{1}{n} \sum_{j = 1}^n \delta_{Y_j} \,.
$$
Given samples $\X$ and $\Y$, the empirical Schr\"odinger coupling is written
 \begin{equation}
     \label{eqn:pi_n_defn}
     \pi_n := \argmin_{\pi \in \Pi(\mu_n, \nu_n)}
  \Big\{ 
 \int \|x-y\|^2\ud \pi (x,y) + \frac{1}{\eta}\KL(\pi \| \mu_n \otimes \nu_n)\Big\}\,.
 \end{equation}

In addition to the computational and conceptual practicalities
of Schr\"odinger couplings, recent work has
shown that they may also avoid the statistical curse of dimensionality endemic
to unregularized OT. 
Indeed, Genevay et. al~\cite{genevay2019sample} (see also~\cite{mena2019statistical})
established the remarkable result
that the cost of $\pi_n$ in~\eqref{eqn:pi_n_defn}
converges to that of $\pi_\star$ in~\eqref{eqn:pi_star_defn}
with the dimension-free statistical rate $1/\sqrt{n}$, in stark contrast
to the unregularized analog which has statistical rate $n^{-1/d}$~\cite{niles2019estimation}. Implicit to these works is the parametric convergence of various quantities, including empirical dual potentials that converge to their population counterparts at a $1/\sqrt{n}$ rate~\cite{mena2019statistical,luise2019sinkhorn, BarGonLou22}.

In this work, we study the sample
complexity of entropic optimal transport for
initial tasks such as cost and density estimation as
well as downstream applications like regression.
Our first set of results improves upon the previous
works for entropic OT cost estimation, and also establishes
rates for the
map as well as the density of $\pi_\star$ with respect to $\mu \otimes \nu$. 
The proofs proceed via an elementary and direct analysis of the empirical
dual problem. From this, we study the stochastic process $(\pi_n-\pi_\star)(\phi)$, indexed by bounded functions $\phi$
and apply these ideas to yield fast $1/n$ rates
for a family
of transfer learning problems. 
A major feature of our proofs is their simplicity; we obtain
all of our results without resorting to technical empirical process
arguments, in contrast to most of the previous statistical literature on both regularized and unregularized OT.

Central limit theorems for the convergence of $\pi_n$ to
$\pi_\star$ have been studied by~\cite{KlaTamMun20} in the discrete support case and partially in~\cite{gunsilius2021matching} for more general cases; see also~\cite{ghosal2021stability} for general consistency results.
Quantitative versions of this result, with
differences measured in $W_1$, appeared in both~\cite{eckstein2021quantitative, deligiannidis2021quantitative}.
The latter quantitative results are similar to our own
but because they measure stability in 
the $W_1$ distance, they give rates which suffer from the
curse of dimensionality.
Our proofs, however, are different and notably simpler. Moreover, they do not require  any smoothness assumptions on the cost function.
This paper is organized as follows. We give an introduction to entropic optimal transport in section~\ref{subsec:entropic_ot_prelim} and define the main quantities of interest. Our main results, namely parametric rates of convergence for the estimation of these quantities are presented in  section~\ref{sec:main_results} together with applications to transfer learning. Most proofs are postponed to section~\ref{sec:sample_complexity}, except for some technical proofs, including those of high-probability bounds, which are relegated to the appendix.

\medskip 
\noindent {\sc Notation.} For an integer $m$, we set $[m] := \{1, \ldots, m\}$,
and let $\Sigma_m$ denote the set of permutations on $[m]$.
A norm $\|\cdot\|$ without subscript denotes the standard
Euclidean distance. 
Given a probability measure $\beta$ on $\R^m$, 
and a $\beta$-integrable function $f \colon \R^m \to \R^k$,
we often abbreviate $\beta(f) := \int f \ud \beta$ or even  $\E[f(X)]$ when $\beta$ is clear from the context. When $\beta$ and $f$ are defined on a product space, we may also write $\beta(f(x,y))$ for $\int f(x,y) \ud \beta(x,y)$. 
Moreover, for any $q \ge 1$, we put
$$
\|f\|_{L^q(\beta)} := \Big(\int \|f(x)\|^q \ud \beta(x) \Big)^{1/q}\, .
$$ The space $L^2(\beta)$ is the Hilbert space (modulo equivalence under
this norm) of $f$
for which the above norm is finite
and with
inner product written $\langle \cdot, \cdot\rangle_{L^2(\beta)}$.
By $\|f\|_{L^{\infty}(\beta)}$ we mean the essential supremum 
$\esup\|f(X)\|$ for $X \sim \beta$.
We denote by $L^\infty(\beta)$ the space of functions $f$ such that  $\|f\|_{L^{\infty}(\beta)}<\infty$.

Given two vectors $a, b \in \R^d$,
we denote their concatenation $(a, b) \in \R^{2d}$,
which is the vector with first $d$ components equal to $a$
and last $d$ components equal to $b$.
Given probability distributions $\beta_0$ on $\R^{m_0}$ and
$\beta_1$ on $\R^{m_1}$, we denote the product
$\beta_0 \otimes \beta_1$, which is defined, for each Borel set
$B \subset \R^{m_0} \times \R^{m_1}$, as
$$
(\beta_0 \otimes \beta_1)(B) := \int \beta_0(B^y) \ud \beta_1(y)
 = \int \beta_1(B_x) \ud \beta_0(x)\, ,
$$ where $B^y := \{x \in \R^{m_0} \colon (x, y) \in B\}$
and $B_x :=\{y \in \R^{m_1} \colon (x, y) \in B\}$.
We denote $k$-fold products as $\beta^{\otimes k}$.

For $\theta \in \R^m$, let $\supp(\theta) \subset [m]$
denote the coordinates on which $\theta$ is non-zero.
For each $\theta \in \R^m$ and $S \subset [m]$,
let $\theta_S := (\theta_k \mathbbold{1}[k \in S])_{k = 1}^m$,
namely the vector with coordinates not in $S$ set to $0$.
A continuously differentiable function $\rho \colon \cH \to \R$, where $\cH$ is a Hilbert space equipped with norm $\|\cdot\|$
is said to be $\alpha$-strongly convex with respect to
the norm $\|\cdot\|$ on a convex
set $C \subset \cH$ if, for all $u, v \in C$,
$$
\rho(v) \geqslant \rho(u) + \langle \nabla \rho(u), v - u \rangle + \frac{\alpha}{2}\|u - v\|^2 \, .
$$ It is said to be $\alpha$-strongly
concave on $C$ if $-\rho$ is $\alpha$-strongly convex on $C$.  

We frequently use the notation $A \lesssim B$ to
write the inequality $A \leqslant c B$ where
$c = c(\eta)$ is a constant depending only on $\eta$.
These suppressed constants
are typically exponentially large in $\eta$.
While the exact dependence of our results in $\eta$ can be easily extracted from our proofs, we have made no attempt to optimize it since we think of $\eta$ as a constant of order 1 throughout. We leave to future work the interesting direction
of improving the dependence of our statistical
results on $\eta$.

\section{Preliminaries on entropic optimal transport}\label{sec:preliminaries_ent_OT}

As stated in the introduction, we assume throughout  that $\mu$ and $\nu$ have bounded support. By re-scaling 
the entropic OT objective, we make the following assumption.
\begin{assumption}\label{assump:bded_support} Assume for $\mu$-almost every $x$ and
$\nu$-almost every $y$,
$\|x\| \leqslant 1/2$ and $\|y\|\leqslant 1/2$.
\end{assumption}

Our main results pertain to specific quantities arising in optimal transport. While these are natural and straightforward, their introduction requires a bit of additional notation. In this section, we introduce our main quantities of interest: dual potentials, cost function, density, and a map known as the entropic regression function.  Empirical counterparts for these objects are functions defined only data points and this section ends with the description of a canonical way to extend them over the entire space $\R^d$ (or $\R^d \times \R^d$).

We begin in section~\ref{subsec:entropic_ot_prelim} by providing background and notation for the dual entropic optimal transport problem that forms the foundation of our results. Thus equipped, we introduce in section~\ref{subsec:quantities} the main quantities of interest, as well as their empirical counterparts, which serve as a basis for our estimators. 
In section~\ref{subsec:canonical_extension} we introduce the
aforementioned canonical extensions.

\subsection{Duality theory for entropic optimal transport}
\label{subsec:entropic_ot_prelim}
The following results, as well as a discussion of the literature on duality
for the entropic optimal transport problem, can be found in~\cite{di2020optimal}.

\begin{theorem}[Strong duality]\label{thm:entropic_ot_main}
Let $P$ and $Q$ be distributions on $\R^d$ with bounded support and fix $\eta >0$.
Denote the entropic optimal transport problem
\begin{equation}\label{eqn:entropic_OT_primal}
S(P, Q) := \inf_{\pi \in \Pi(P,Q)}\big\{ \pi(\|x - y\|^2) + \frac{1}{\eta}\KL(\pi\| P \otimes Q)\big\}\,.
\end{equation} The infimum is attained by a unique $\pi_\star \in \Pi(P, Q)$ and strong duality holds in the sense that
\begin{equation}\label{eqn:entropic_OT_strong_duality}
S(P, Q)
= \sup_{(f, g) \in L^{\infty}(P) \times L^{\infty}(Q)} \Big\{ P(f) + Q(g) - \frac{1}{\eta}(P \otimes Q)
(e^{-\eta\|x - y\|^2 + \eta f(x) + \eta g(y)} - 1) \Big\}\,.
\end{equation} The supremum above is attained at a pair $(f_0, g_0) \in L^{\infty}(P) \times L^{\infty}(Q)$
of dual potentials, which are unique up to the translation
$(f_0, g_0) \mapsto (f_0 + c, g_0 - c)$ for $c \in \R $.

Moreover, primal and dual solutions are linked via the following relationships. For any pair $(f, g) \in L^{\infty}(P) \times L^{\infty}(Q)$,
let $\pi$ be the measure with density
\begin{equation}\label{eqn:PQ_scaling}
\frac{\ud\pi}{\ud(P \otimes Q)}(x, y) = e^{-\eta\|x -y \|^2 + \eta f(x) + \eta g(y)}\,.
\end{equation}
Then the pair $(f, g)$ is optimal for~\eqref{eqn:entropic_OT_strong_duality} if and only if
 $\pi\in \Pi(P, Q)$ and $\pi$ is optimal for~\eqref{eqn:entropic_OT_strong_duality}. 
\end{theorem}

We denote the population dual objective by
\begin{equation}\label{eqn:pop_dual_objective}
\Phi(f, g) := \mu(f) + \nu(g) - \frac{1}{\eta} (\mu \otimes \nu)\big( e^{-\eta(\|x - y\|^2 -
f(x) - g(y))} \big) + \frac{1}{\eta}\,.
\end{equation}
To account for translation invariance of dual potentials, we distinguish the unique optimal dual potentials $(f_\star, g_\star)$ such that
$\nu(g_\star) = 0$. 

Throughout our proofs, we make extensive use of another optimal pair of dual potentials denoted $(\bar f_\star, \bar g_\star)$ defined by $\bar f_\star=f_\star+ \nu_n (f_\star)$ and $\bar g_\star=g_\star - \nu_n(g_\star)$ so that $\nu_n(\bar g_\star)=0$.

The empirical dual objective $\Phi_n$ is defined by
\begin{equation}\label{eqn:emp_dual_objective}
\Phi_n(f, g) := \mu_n(f) + \nu_n(g) - \frac{1}{\eta} (\mu_n \otimes \nu_n)(e^{-\eta(\|x - y\|^2 - f(x)
- g(y))}) + \frac{1}{\eta}\,.
\end{equation}
As in the population case, we distinguish the unique
optimizers $(f_n, g_n)$ such that $\nu_n(g_n) = 0$.

In the sequel, we will be using the gradient of $\Phi_n$ that are defined as elements of (the dual of) $L^2(\mu_n) \times L^2(\nu_n)$ and given by 
\begin{align}
  \langle \nabla \Phi_n (f,g), (\varphi,\psi)\rangle_{L^2(\mu_n)\times L^2(\nu_n)}&=(\mu_n\otimes \nu_n)\big((\varphi(x)+\psi(y)) (1-e^{-\eta(\|x - y\|^2 -
f(x) - g(y))})  \big)\,.   \label{eq:gradphin}
\end{align}
In particular, we readily get the following expression for the norm of the above gradient:
\begin{equation}
    \label{eq:norm_gradphin}
    \|\nabla \Phi_n (f,g)\|_{L^2(\mu_n) \times L^2(\nu_n)}^2 =  \frac1n\sum_{i=1}^n \left(1-\frac1n\sum_{j=1}^n p(X_i,Y_j)\right)^2+ \frac1n\sum_{j=1}^n \left(1-\frac1n\sum_{i=1}^n p(X_i,Y_j)\right)^2\,,
\end{equation}
where $p$ is the density defined in~\eqref{eqn:PQ_scaling} and is given by $p(x,y)=e^{-\eta\|x-y\|^2+\eta f(x) + \eta g(y)}.$

\begin{remark}
Throughout this paper we make extensive use of the fact that $\Phi, \Phi_n,$ and $\nabla \Phi_n$ enjoy the following translation invariance property: For any constant $c \in \R$, it holds
$$
\Phi(f,g)=\Phi(f+c, g-c)\,, \quad \Phi_n(f,g)=\Phi_n(f+c,g-c)\,, \quad \text{and} \quad \nabla \Phi_n(f,g)=\nabla \Phi_n(f+c,g-c)\,.
$$
\end{remark}
\subsection{Quantities of interest}
\label{subsec:quantities}

Beyond the couplings $\pi_\star$ and $\pi_n$ that have already been defined, we now introduce important quantities associated to entropic optimal transport: cost, density, and map.

\paragraph*{Cost.} Using~\eqref{eqn:entropic_OT_primal}, define the population entropic OT cost as $S:=S(\mu, \nu)$ and its plug-in estimator $S_n:=S(\mu_n, \nu_n)$.

\paragraph*{Density.} Define the density of the optimal coupling $\pi_\star$ with respect to the product measure $\mu \otimes \nu$ as
\begin{equation}\label{eqn:pop_scaling}
p_\star(x, y) := \frac{\ud\pi_\star}{\ud(\mu \otimes \nu)}(x,y) = e^{-\eta(\|x - y\|^2 - f_\star(x) - g_\star(y))}\,.
\end{equation} The statement $\pi_\star \in \Pi(\mu, \nu)$ can then
be written in the following succinct manner.
For $\mu$-almost every $x$ and $\nu$-almost every $y$,
\begin{equation}\label{eqn:pop_marg}
\nu(p_\star(x,\cdot))= \mu(p_\star(\cdot, y))=1\,.
\end{equation}
Equivalently,
\begin{align}
    f_\star(x)&=-\frac1\eta \ln \Big(\int e^{-\eta\|x-y\|^2 + \eta g_\star(y)} \ud \nu(y)\Big) \label{eqn:fstar_from_gstar}\\
     g_\star(y)&=-\frac1\eta \ln\Big(\int e^{-\eta\|x-y\|^2 + \eta f_\star(x)} \ud \mu(x)\Big)\label{eqn:gstar_from_fstar}
\end{align}

Its empirical counterpart is the density of $\pi_n$ with respect to the product measure $\mu_n \otimes \nu_n$. It is defined for all $x \in \X$ and
$y \in \Y$, as 
\begin{equation}\label{eqn:emp_decomp}
p_n(x, y) := \frac{\ud\pi_n}{\ud(\mu_n \otimes \nu_n)}(x,y) 
= e^{-\eta(\|x - y\|^2 - f_n(x) - g_n(y))}.
\end{equation}
The marginal constraints then become, for all $x \in \X$ and $y \in \Y$,
the equalities
\begin{equation}\label{eqn:emp_marg}
\nu_n(p_n(x,\cdot))= \mu_n(p_n(\cdot, y))=1\,.
\end{equation}

\paragraph*{Map.}
Finally, we consider the entropic
analog of the optimal transport map, namely the \emph{entropic regression function}---sometimes called ``barycentric projection"---defined as the map
\begin{equation}\label{eqn:cond_pistar}
b_\star(x) := \E_{\pi_\star}[Y\, | \, X=x]\,.
\end{equation}
Its empirical counterpart is the plug-in estimator
\begin{equation}\label{eqn:cond_pin}
b_n(x) := \E_{\pi_n}[Y \, | \, X=x]\,.
\end{equation} The estimator $b_n$ was recently proposed in~\cite{pooladian2021entropic} as a computationally efficient surrogate for a smooth Brenier map, and rates of estimation were established, albeit suboptimal ones. 

\subsection{Canonical extensions}\label{subsec:canonical_extension}
The marginal constraints in~\eqref{eqn:emp_marg}
induce a canonical extension of the optimal dual potentials
$(f_n, g_n)$ to maps defined
on all of $\R^d$~\cite{berman2020sinkhorn,pooladian2021entropic}.
For example, the marginal constraint for $x \in \X$ implies
\begin{equation}\label{eqn:fn_canonical}
e^{-\eta f_n(x)} = \frac{1}{n} \sum_{j = 1}^n e^{-\eta\|x - Y_j\|^2 + \eta g_n(Y_j)}.
\end{equation} Since the right-hand side makes sense for any $x \in \R^d$,
we thus abuse notation and define $f_n(x)$, for any $x \in \R^d$,
to satisfy this equation. We similarly define $g_n(y)$,
for any $y \in \R^d$, to satisfy
\begin{equation}\label{eqn:gn_canonical}
e^{-\eta g_n(y)} = \frac{1}{n} \sum_{i = 1}^n e^{-\eta\|X_i - y\|^2 + \eta f_n(X_i)}.
\end{equation} 
In turn, this canonical extension also applies to the density   $p_n$ in a straightforward manner as well as to the entropic regression function estimator $b_n$. From now on, we employ the definition
\begin{equation}
    \label{eq:canonical_extension_bn}
    b_n(x)=\frac{\sum_{i=1}^n Y_i p_n(x,Y_i)}{\sum_{i=1}^n p_n(x,Y_i)}=\frac1n\sum_{i=1}^n Y_i p_n(x,Y_i)\,, \quad x \in \R^d\,.
\end{equation}

\section{Main results}\label{sec:main_results}

In this section, we state our main results on the rates of estimation of the quantities introduced in section~\ref{subsec:quantities}. We also present our main application to transfer learning.

\subsection{Sample complexity}\label{subsec:sampling_complexity_main}
In this section we give our main results on rates of estimation for various quantities arising from entropic optimal transport: the cost,  the density of the coupling, the entropic regression function, as well as the coupling itself.

\paragraph*{Cost.}
Our first result concerns the rate of convergence
of $S_n$ to $S$.
We emphasize that
the unregularized analog of this quantity $(\eta \to \infty)$ is generically of order
at least $n^{-2/d}$, and, in fact, it is known that no estimator
for the unregularized cost
can beat the rate $n^{-1/d}$ in general~\cite{niles2019estimation}.
For the entropic problem, existing results imply
$\E|S_n - S| \lesssim 1/\sqrt{n}$
under mild assumptions on $\mu$ and $\nu$~\cite{genevay2019sample,luise2019sinkhorn,mena2019statistical,BarGonLou22}.
Our techniques yield a commensurate bound in a stronger sense: a $1/n$ rate both for the mean squared error and the bias using a different proof. The result stating that the bias is an order of magnitude smaller than the stochastic fluctuations is quite remarkable and seems to have appeared only in the concurrent work~\cite{BarGonLou22}.
\begin{theorem}\label{thm:sample_complexity_cost}
The mean squared error and bias are bounded respectively as
$$
\E |S_n - S|^2 \lesssim \frac{1}{n}\,, \qquad
|\E [S_n] - S| \lesssim \frac{1}{n}\,.
$$
Moreover, for all $t > 0$, with probability at least $1 - 6e^{-t}$,
$$
|S_n - S| \lesssim \frac{t}{n} + \sqrt{\frac{t}{n}}\,.
$$
\end{theorem}
This result is achieved by a careful analysis of
the empirical dual
problem, and is elementary in that it
essentially only involves straightforward calculations and
standard inequalities arising from strong concavity.
In contrast to most of the
previous literature on statistical estimation problems
in both unregularized and
entropic optimal transport, we completely circumvent the control of suprema of empirical processes.

\paragraph*{Map.}
Our second result concerns the problem of estimating the entropic regression function defined in~\eqref{eqn:cond_pistar}. To that end, we employ the canonical extension $b_n$ defined in~\eqref{eq:canonical_extension_bn}.

In this work, we consider $b_n$ as an estimator not of the Brenier map but of the entropic regression function $b_\star$.  In contrast to the unregularized
case, the computationally feasible estimator $b_n$
achieves the parametric rate $1/n$ for estimating $b_\star$
in arbitrary dimension.

\begin{theorem}\label{thm:sample_complexity_map}
Let $b_\star$ and $b_n$ be as in~\eqref{eqn:cond_pistar}
and~\eqref{eqn:cond_pin}, respectively. Then
$$
\E\|b_n - b_\star\|_{L^2(\mu)}^2 \lesssim \frac{1}{n}\,.
$$
Moreover, for all $t > 0$, with probability at least
$1 - 8e^{-t}$, it holds
$$
\|b_n - b_\star\|_{L^2(\mu)}^2 \lesssim \frac{t}{n}\,.
$$
\end{theorem}

\paragraph*{Density.}
Our third result,
proved in almost exactly the 
same way as our result on entropic
map estimation,
gives a comparable
result for the problem
of density estimation.
\begin{theorem}\label{thm:sample_complexity_density} Let
$p_\star$ be as in~\eqref{eqn:pop_scaling}.
Extend the empirical
analog $p_n$
from~\eqref{eqn:emp_decomp}
as in section~\ref{subsec:canonical_extension}.
Then
$$
\E \|p_n - p_\star\|^2_{L^2(\mu \otimes \nu)} \lesssim \frac{1}{n}\,.
$$
Moreover, for all $t > 0$,
with probability at least $1 - 16e^{-t}$, it holds
$$
\|p_n - p_\star \|_{L^2(\mu_n \otimes \nu_n)}^2, \|p_n - p_\star\|_{L^2(\mu \otimes \nu)}^2
\lesssim \frac{t}{n}.
$$
\end{theorem}
Note that the high-probability bound above holds in both $L^2(\mu_n \otimes \nu_n)$
as well as $L^2(\mu \otimes \nu)$; while the latter is more natural from a statistical estimation perspective,
the former is ultimately more useful for our subsequent applications.
In particular, this result readily implies Theorem~\ref{thm:sample_complexity_map}
and is achieved with essentially the same ideas as our
estimate on the cost, and so has a simple, easy to understand proof.

In sum, the previous three results show that the statistical picture for
entropic optimal transport is fundamentally different than
in unregularized optimal transport: fast, parametric, rates
hold in arbitrary dimension, and they can be established with simple proofs.

\paragraph*{Coupling.}
While these results provide strong evidence that
entropic optimal transport is fundamentally tractable, both
statistically and computationally,
there remains basic questions about the validity
of using $\pi_n$ as a plug-in for the coupling $\pi_\star$ in more complex statistical
procedures. There, we generally
wish to understand the deviations $\pi_\star - \pi_n$.
The next result, proved in section~\ref{subsec:main_conv},
makes a first step in this direction.
\begin{theorem}\label{thm:main_conv}
Let $\phi \in L^{\infty}(\mu \otimes \nu)$. Then, for all $t > 0$,
with
probability at least $1 - 18e^{-t^2}$,
\begin{equation}\label{eqn:subg_pistar_pin}
   |(\pi_\star - \pi_n)(\phi) |
   \lesssim \frac{\|\phi\|_{L^{\infty}(\mu \otimes \nu)}\cdot t}{\sqrt{n}}\,.
\end{equation}
\end{theorem}
In particular, the random variable $(\pi_\star - \pi_n)(\phi)$ is subGaussian with variance proxy of order $1/n$ as in standard empirical process theory even though $\pi_n$ has a dependence in all the observations that is much more complex than a simple average of independent random variables.

In the next section, we refine this result to achieve parametric rates for
a broad family of transfer learning tasks based on entropic optimal transport.

\subsection{Application to transfer learning}\label{subsec:transfer_learning_main}

Recall that in Theorem~\ref{thm:sample_complexity_map} we examined the
sample complexity of estimating the entropic regression function $b_\star(x)= \E_{\pi_\star}[Y\, | \, X=x]$.
This regression function is the solution to the following least squares problem:
\begin{equation}\label{eqn:pop_map_problem}
b_\star = \argmin_{h \in L^2(\mu)}  \E_{\pi_\star}\|h(X) - Y\|^2\,.
\end{equation} 
Similarly, $b_n$ is its empirical counterpart in the sense that
\begin{equation}\label{eqn:emp_map_problem}
b_n= \argmin_{h \in L^2(\mu_n)}  \E_{\pi_n}\|h(X) - Y\|^2\,.
\end{equation}
Theorem~\ref{thm:sample_complexity_map}
implies that the solutions of these two problems are close. The goal of this section is to explore how general this
phenomenon is.
To that end, we consider the more general setup, where in addition to the samples
$\X = (X_1, \ldots, X_n)$ and $\Y =(Y_1, \ldots, Y_n)$,
we observe real-valued labels $A_1, \ldots, A_n \in \R$
such that $A_i \sim q( \cdot\, | \, Y_i)$ for some unknown $q$. As a result, our observations consist of unlabeled observations $X_1, \ldots, X_n$, as well as labeled observations $(Y_1, A_1), \ldots, (Y_n, A_n)$ but the coupling between the labeled and unlabeled observations is not observed. The goal is to predict $A$ from $X$.  Without
making any assumptions on the coupling between the $X$ and $Y$ variables, this problem is clearly impossible. We thus make the additional assumption that the joint distribution of $(X, Y, A)$, denoted by $\omega$, is given by $(X, Y) \sim \pi_\star$
and $A \sim q(\cdot\, |\, Y)$.

A relevant benchmark arises when observations consist of $n$ independent triples $(X_i, Y_i, A_i)$, $i=1, \ldots, n$ where $(X_i, Y_i)\sim \pi_*$ and, conditionally on $Y_i$,  $A_i \sim q(\cdot|Y_i)$. In this case, the Schr\"odinger coupling $\pi_\star$ need not be learned from uncoupled data. One of the main statistical messages of this section is that the cost of estimating the Schr\"odinger coupling is at most the same as the underlying statistical task and thus does not affect the statistical rates.
To illustrate this fact, we investigate two statistical questions: regression and classification. 

\paragraph{Regression with a squared loss.} 
In the transfer learning model above, the regression problem~\eqref{eqn:pop_map_problem} generalizes to
\begin{equation}\label{eqn:pop_npls_problem}
h_\star = \argmin_{h \in L^2(\mu)} F(h):= \E|h(X) - A|^2\,,
\end{equation} 
so that
\begin{equation}\label{eqn:def_h_star}
h_\star(x)=\E_{\omega}[A \, | \, X = x] = \E_{\omega}[Ap_\star(x, Y)]\,.
\end{equation}
Analogously, define
\begin{equation}\label{eqn:emp_npls_problem}
 h_n =   \argmin_{h \in L^2(\mu_n)} \hat F(h)  :=  \frac{1}{n^2}\sum_{i, j = 1}^n p_n(X_i, Y_j) (h(X_i) - A_j)^2\,,
\end{equation}
where $p_n$ is the density defined in~\eqref{eqn:emp_decomp}.
We extend $h_n$ to an element of $L^2(\mu)$ by
using the canonical extension of $p_n$ in the formula for
$h_n$; that is, for all $x \in \R^d$, we set
\begin{equation}\label{eqn:emp_npls_soln}
h_n(x) := \frac{1}{n} \sum_{j = 1}^n A_j p_n(x, Y_j) \quad \quad 
x \in \R^d\,.
\end{equation}
Using this equation and taking successively $A=u_k^\top Y$,
where $u_1, \ldots, u_d$ is an orthonormal basis for $\R^d$
and $k \in [d]$,
it is easy to see that this model generalizes that of
equations~\eqref{eqn:pop_map_problem} and~\eqref{eqn:emp_map_problem}.
Incorporating general labels $A$
allows us to model a common
practical situation where we are
given two datasets, one labeled and the other not,
and we wish to transfer the labels from one to the other.

The next theorem is a direct extension of Theorem~\ref{thm:sample_complexity_map} to the more general setup with labels. Its proof is postponed to subsection~\ref{sec:pr_transfer} in the appendix.

\begin{theorem}\label{thm:sample_complexity_npls}
Suppose that $\omega$-almost surely, $|A| \leqslant 1$
and $h_\star, h_n$ are defined
as in~\eqref{eqn:pop_npls_problem} and~\eqref{eqn:emp_npls_soln},
respectively. Then
$$
\E \|h_\star - h_n\|_{L^2(\mu)}^2 \lesssim \frac{1}{n}\,.
$$ 
\end{theorem}
A careful inspection of its proof indicates that one may achieve not only rates of estimation in $L_2(\mu)$ but also pointwise guarantees. These are instrumental in the analysis of plug-in classifiers that we carry out now.

\paragraph{Plug-in classification.} 
Assume now that $A \in \{0,1\}$ and the goal is to predict $A$ from $X$ using a classifier $h: \R^d \to \{0,1\}$. We measure the performance of a classifier $h$ using the classification error $\p[h(X)\neq Y]$. It is well known that the best classification error is achieved by the Bayes classifier $\ell_*(x)=\1(h_\star(X)>1/2)$, where $\1(\cdot)$ denotes the indicator function and $h_\star$ is the regression function defined in~\eqref{eqn:def_h_star}. This leads us to consider the excess-risk of a classifier $\ell$, which is defined as~\cite{DevGyoLug96}
$$
\cE(\ell)= \p[\ell(X)\neq Y] -  \p[\ell_\star(X)\neq Y] \ge 0\,.
$$

This explicit form of the Bayes classifier together with the availability of a good nonparametric estimator $h_n$ defined in~\eqref{eqn:emp_npls_soln} suggest a natural \emph{plug-in} classifier $\ell_n$ defined by
$$
\ell_n(x) =\1(h_n(x)>1/2)\,.
$$
Such estimators were investigated in~\cite{AudTsy07} under the classical Mammen-Tsybakov noise condition which controls the steepness of the function $h_\star$ around value $1/2$. We recall its definition here for convenience.

\begin{defn}[Mammen-Tsybakov noise condition \cite{MamTsy99}]  We say that the distribution $\omega$ satisfies the Mammen-Tsybakov condition with parameter $\alpha>0$, if there exists a constant $C_0>0$ such that $\mu(0<|h_\star-1/2|\le \eps )\le C_0 \eps^\alpha$ for all $\eps \in (0,1/2]$.
\end{defn}
Note that this condition governs simultaneously how steeply $h_\star$ crosses level $1/2$ and how much $\mu$ puts mass around the noise region $\{x:h_\star(x)=1/2\}$. For this reason, it is also referred to \emph{margin} or \emph{low noise} condition.

In their seminal paper~\cite{AudTsy07} establish a key lemma, which relates the excess-risk of a plug-in estimator $\ell(x)=\1(h(x)>1/2)$ to that of the pointwise estimation error of $h_\star$ by $h$; see~\cite{AudTsy07}, Lemma~3.1.
Combining this result with pointwise guarantees for $h_n$, we establish the following result.
\begin{theorem}
\label{thm:classif}
Suppose that $\omega$ satisfies the Mammen-Tsybakov condition with parameter $\alpha>0$. Then the excess risk of the plug-in classifier $\ell_n$ defined above has excess risk bounded as
$$
\E[\cE(\ell_n)]\lesssim n^{-\frac{1+\alpha}{2}}\,.
$$
\end{theorem}

Note that the rate of Theorem~\ref{thm:classif} is at least as fast as $1/\sqrt n$ for when $\alpha=0$. As $\alpha\to \infty$, it can be easily verified that it leads to exponential rates as in~\cite{KolBez05}.

\section{Proofs }
\label{sec:sample_complexity}

In this section we give the proofs for our convergence results in expectation and postpone the proofs of tail bounds to the appendix. Indeed, the
proofs in expectation are simpler and already contain the key ingredients to obtain the parametric rates of convergence showcased in section~\ref{sec:main_results}.

\subsection{Structural results on entropic optimal transport}\label{sec:preliminaries}
We begin with some key structural results---chiefly  strong concavity for the empirical dual problem---which drive the parametric rates of convergence. The strong concavity
result relies on boundedness of the dual potentials which follows from the bounded support assumption~\ref{assump:bded_support} and is established by the following proposition,
proved in section~\ref{subsec:bded_dual}.

\begin{prop}[Bounded dual potentials and densities]\label{prop:bded_dual}
As above, specify the unique dual potentials $(f_n, g_n)$
and $(f_\star, g_\star)$ such that
$\nu_n(g_n) = \nu(g_\star) = 0$, and extend $(f_n, g_n)$
according to~\eqref{eqn:fn_canonical} and~\eqref{eqn:gn_canonical},
respectively.
Then
$$
\| f_n\|_{L^{\infty}(\mu)},\, \, \|g_n \|_{L^{\infty}(\nu)} \le 2\,, \qquad 
 \|f_\star \|_{L^{\infty}(\mu)}, \, \, \|g_\star\|_{L^{\infty}(\nu)} \leqslant 1\,.
$$
In particular, for $(\mu \otimes \nu)$-almost every $(x, y)$,
\begin{equation}\label{eqn:unif_bd}
e^{-5\eta} \leqslant p_\star(x,y), \, \, p_n(x,y ) \leqslant e^{5\eta }\,.
\end{equation}
\end{prop}
We remark that, unfortunately, this exponential dependence on $\eta$
is unavoidable in the worst case. For an example,
see~\cite[section 3]{altschuler2022asymptotics}.

Recall that we denote
the empirical dual objective by $\Phi_n$, as in~\eqref{eqn:emp_dual_objective}.
The next lemma simply says that $\Phi_n$ is strongly concave so long
as we remove the symmetry $(f, g) \mapsto (f + c, g - c)$ using the condition $\nu(g)=0$ and consider only uniformly
bounded potentials. More precisely, we work on the convex set of dual potentials
$$
\cS_L := \{(f ,g) \in L^{\infty}(\mu_n)
\times L^{\infty}(\nu_n) \colon \|f\|_{L^{\infty}(\mu_n)} \lor \|g\|_{L^{\infty}(\nu_n)} \leqslant L,  \, \, \nu_n(g) = 0 \}\,.
$$
In this section, we use the shorthand notation $\langle\cdot, \cdot\rangle_n$ and $\|\cdot\|_n$ respectively to denote the inner product and norm of $L^2(\mu_n) \times L^2(\nu_n)$.

We are now in a position to state our first structural result.
   
\begin{lemma}[Strong concavity of the empirical dual]\label{lem:emp_dual_sc}
For each $L > 0$, $\Phi_n$ is $\delta$-strongly concave with respect to the norm $\|\cdot\|_{n}$
on $S_L$
 for $\delta = \eta e^{-\eta (2L + 1)}$ in the sense that for any $(f,g), (f',g') \in \cS_L$, we have almost surely
$$
\Phi_n(f,g)-\Phi_n(f',g')\ge \langle \nabla \Phi_n(f,g), (f,g)-(f',g') \rangle_n + \frac{\delta}{2}\|(f,g)-(f',g')\|^2_n\,.
$$
\end{lemma}

\begin{proof} Fix two pairs $(f, g), (f' , g') \in \cS_L$.
Then it suffices to show that the function 
$$
h(t) := \Phi_n((1 - t)f + t f', (1 - t)g + tg')\,,
$$ satisfies
$$
h''(t) \leqslant -\delta \|(f, g) - (f', g')\|^2_n \, ,\quad \quad 
\forall  t\in [0,1]\,.
$$ Fix some $t \in [0,1]$. Observe that the linear terms cancel so that
\begin{align*}
h''(t) = -\frac{\eta}{n^2} \sum_{i, j = 1}^n& (f(X_i) - f'(X_i)
+ g(Y_j) - g'(Y_j))^2\\
&\times\exp \big(-\eta \|X_i - Y_j\|^2 +  \eta t( f(X_i) +
   g(Y_j) ) + \eta  (1 - t)(f'(X_i) +  g'(Y_j)) \big)\,.
\end{align*}
Using our bounded support Assumption~\ref{assump:bded_support},
along with the definition of $\cS_L$, yields
$$
h''(t) \leqslant - \eta e^{-\eta (2L + 1)} \cdot \frac{1}{n^2}
\sum_{i, j = 1}^n (f(X_i) - f'(X_i)
+ g(Y_j) - g'(Y_j))^2\,.
$$ Expanding out these squared terms, we use the fact that
$\nu_n(g) = \nu_n(g') = 0$ to find
\begin{align*}
\frac{1}{n^2} \sum_{i, j = 1}^n (f(X_i) - f'(X_i)
+ g(Y_j) - g'(Y_j))^2 &= \|(f, g) - (f', g')\|_n^2
+ 2\mu_n(f - f')\cdot \nu_n(g - g') \\
&= \|(f, g) - (f', g')\|_n^2.
\end{align*} Hence the result.
\end{proof}

The following consequence
of strong concavity of $\Phi_n$, known as a {Polyak-\L{}ojasiewicz
(PL) inequality}, is instrumental in our results.
\begin{prop}[PL inequality for the empirical dual]\label{prop:emp_objective_PL}
Let $L > 0$ be such that $(f_n, g_n) \in \cS_L$.
Then for any $(f, g) \in \cS_L\,$,
\begin{equation}
    \label{EQ:pl}
    \Phi_n(f_n, g_n) - \Phi_n(f, g) \leqslant 
 \frac{e^{\eta(2 L + 1)}}{2\eta} \| \nabla \Phi_n(f, g)\|_{n}^2\,.
\end{equation}
\end{prop}
Proposition~\ref{prop:emp_objective_PL} is a weaker form
of strong concavity that has recently been
thoroughly studied in the optimization literature~\cite{karimi2016linear},
but for the reader's convenience
we include a proof in
section~\ref{subsec:emp_objective_PL}. 

In light of Proposition~\ref{prop:bded_dual}, we have that $(\bar f_\star, \bar g_\star):=(f_\star+\nu_n(g_\star),g_\star-\nu_n(g_\star)) \in S_2$ almost surely. On the one hand, together with the PL inequality~\eqref{EQ:pl} it yields
\begin{align*}
    \Phi_n(f_n, g_n) - \Phi_n(\bar f_\star, \bar g_\star)
    &\leqslant 
 \frac{e^{5\eta}}{2\eta} \| \nabla \Phi_n(f_\star+\nu_n(g_\star), g_\star-\nu_n(g_\star))\|_{n}^2= \frac{e^{5\eta}}{2\eta} \| \nabla \Phi_n(f_\star, g_\star)\|_{n}^2\,.
\end{align*}
On the other hand, strong concavity in Lemma~\ref{lem:emp_dual_sc} gives
\begin{align*}
     \Phi_n(f_n, g_n) - \Phi_n(\bar f_\star, \bar g_\star)
  &\geqslant 
 \frac{\eta e^{-5\eta}}{2}   \|(f_n - \bar f_\star, g_n - \bar g_\star)\|_{n}^2\,.
\end{align*}
The above two displays imply
\begin{equation}
\label{eq:EB}
  \|(f_n - \bar f_\star, g_n - \bar g_\star)\|_{n} \le 
\frac{e^{5\eta}}{ \eta} \| \nabla \Phi_n(f_\star,  g_\star)\|_{n}.
\end{equation}
This inequality has been called ``error bound" in the optimization literature and can, in fact, be shown to be equivalent to a PL inequality; see~\cite[Theorem~2]{karimi2016linear}. As detailed in the next section, inequality~\eqref{eq:EB}  readily yields dimension-independent rates of estimation for the population dual potentials $(f_\star, g_\star)$. 

We conclude with a tool that becomes useful when  assessing the quality of the canonical extensions employed in this paper.
It is essentially a consequence of
the Lipschitzness
of the exponential and logarithm
on bounded intervals. Its proof is postponed to Appendix~\ref{sec:appendix_further}.

\begin{prop}[Comparison inequalities for dual
potentials and densities]\label{prop:lipschitz}
For $\mu$-almost every~$x$,
$$
 |\bar f_\star(x)  - f_n(x)| \lesssim 
\Big| \frac{1}{n} \sum_{j = 1}^n e^{-\eta \|x - Y_j\|^2 + \eta  g_\star(Y_j)} - 
   \nu(e^{-\eta \|x - \cdot\|^2 + \eta   g_\star(\cdot)} ) \Big| + \|\bar g_\star - g_n\|_{L^1(\nu_n)}\,.
$$ Similarly, for $\nu$-almost every $y$,
$$
|\bar g_\star(y) - g_n(y)| \lesssim 
\Big| \frac{1}{n} \sum_{i = 1}^n e^{-\eta \|X_i - y\|^2 + \eta   f_\star(X_i)} - 
   \mu(e^{-\eta \|\cdot - y\|^2 + \eta  f_\star(\cdot)} ) \Big| + \|\bar f_\star - f_n\|_{L^1(\mu_n)}\,.
$$ And for $(\mu \otimes \nu)$-almost every $(x, y)$,
\begin{equation}
    \label{eqn:pstar-pn}
    |p_\star(x, y) - p_n(x, y)| \lesssim 
|\bar f_\star(x)  - f_n(x)|  + |\bar g_\star(y) - g_n(y)|\,.
\end{equation}
\end{prop}

\subsection{Rates of convergence in expectation}

We begin with an important, albeit simple, argument to establish our convergence results. Central to our proof is the fact that we only need to measure empirical deviations at $(f_\star, g_\star)$, the optimal dual potentials. 
\begin{lemma}
\label{lem:emp_dev_gradphi}
Let $\Phi$ and $\Phi_n$ be the dual objectives defined in~\eqref{eqn:pop_dual_objective} and~\eqref{eqn:emp_dual_objective} respectively. Moreover, let $(f_\star, g_\star)$ be the pair of potentials that maximizes $\Phi$ and such that $\nu(g_\star)=0$. Then
$$
 \E\| \nabla \Phi_n(f_\star, g_\star)\|_{n}^2 \le \frac{2e^{10\eta}}{n}\,,
$$
where $\|\cdot\|_n$ denotes the norm of  $L^2(\mu_n) \times L^2(\nu_n)$.
\end{lemma}
\begin{proof}
Recall that $p_\star$ is defined in~\eqref{eqn:pop_scaling} so by \eqref{eq:norm_gradphin} we get 
\begin{align*}
    \E\| \nabla& \Phi_n(f_\star, g_\star)\|_{n}^2 
    =\frac1n\sum_{i=1}^n \E\Big[1-\frac1n\sum_{j=1}^n p_\star(X_i,Y_j)\Big]^2+ \frac1n\sum_{j=1}^n \E\Big[1-\frac1n\sum_{i=1}^n p_\star(X_i,Y_j)\Big]^2\\
    &=\frac{1}{n^2} \Big(\sum_{j,l=1}^n\mathsf{Cov} \left[1-p_\star(X_1,Y_j),1-p_\star(X_1,Y_l)\right]+ \sum_{i,k=1}^n\mathsf{Cov} \left[1-p_\star(X_i,Y_1),1-p_\star(X_k,Y_1)\right]\Big)\,,
\end{align*}
where $X_i \sim \mu$, $Y_j \sim \nu$ with $X_1, \ldots, X_n$ mutually independent, and $Y_1, \ldots, Y_n$ mutually independent. In particular, in light of the marginal constraints listed in~\eqref{eqn:pop_marg}, it holds that 
$$
 \mathsf{Cov} \left[1-p_\star(X_i,Y_j),1-p_\star(X_k,Y_l) \right]=0\,,\quad \forall \  (i,j)\neq (k,l)\,.
$$
As a result,
\begin{align*}
    \E\| \nabla& \Phi_n(f_\star, g_\star)\|_{n}^2 =\frac{2}{n}\mathsf{Var} \left[1-p_\star(X_1,Y_1)\right]=\frac2{n} \mathsf{Var} \left(p_\star(X_1,Y_1) \right)\le \frac{2e^{10\eta}}{n}\,,
\end{align*}
where in the inequality, we used~\eqref{eqn:unif_bd}. 
\end{proof}

\paragraph*{Bias of cost.}
We here give the proof of the bias inequality part of Theorem~\ref{thm:sample_complexity_cost}.
Recall that the population and empirical cost are denoted, respectively,
$$
S := \pi_\star(\|x - y\|^2) + \frac{1}{\eta}\KL(\pi_\star \| \mu \otimes \nu)\, , \quad 
\quad S_n := \pi_n(\|x - y\|^2) + \frac{1}{\eta} \KL(\pi_n \|\mu_n \otimes \nu_n)\, .
$$
Our goal is to establish the bias bound of Theorem~\ref{thm:sample_complexity_cost}: $|\E[S_n] - S| \lesssim 1/n$.
To that end, observe that
$$
S_n - S = \Phi_n(f_n,g_n) - \Phi(f_\star, g_\star) = \{ \Phi_n(f_n, g_n)
- \Phi_n(f_\star, g_\star) \} + \{ \Phi_n(f_\star, g_\star) - \Phi(f_\star, g_\star)\}\,.
$$ 
Since $\E[\Phi_n]=\Phi$, we get
\begin{equation}\label{eqn:upperexpect}
\E[S_n] - S = \E[\Phi_n(f_n, g_n) -
\Phi_n(f_\star, g_\star)]\,.
\end{equation} Using this equation and optimality of $(f_n, g_n)$,
we can conclude $\E[S_n] \geqslant S$; $S_n$ has a non-negative bias.
Thus, it suffices to focus on upper-bounding the right-hand
side of~\eqref{eqn:upperexpect}. 
To do this, recall first that in light of Proposition~\ref{prop:bded_dual}, we have that $(f_\star+\nu_n(g_\star),g_\star-\nu_n(g_\star)) \in S_2$ almost surely. Therefore, the PL inequality in Proposition~\ref{prop:emp_objective_PL} yields
\begin{align}
    \Phi_n(f_n,g_n)-\Phi_n(f_\star, g_\star)&=  \Phi_n(f_n,g_n)-\Phi_n(f_\star+\nu_n(g_\star), g_\star-\nu_n(g_\star))\\
    &\leqslant 
 \frac{ e^{5\eta}}{2\eta}\| \nabla \Phi_n(f_\star+\nu_n(g_\star), g_\star-\nu_n(g_\star))\|_{L^2(\mu_n \otimes \nu_n)}^2\nonumber\\
&= \frac{ e^{5\eta}}{2\eta} \| \nabla \Phi_n(f_\star, g_\star)\|_{L^2(\mu_n \otimes \nu_n)}^2\, .\label{eq:boundgrad1}
\end{align}
Applying Lemma~\ref{lem:emp_dev_gradphi}, we get
$$
\E [\Phi_n(f_n,g_n)-\Phi_n(f_\star, g_\star)]
\leqslant \frac{e^{15\eta}}{\eta}\cdot \frac{1}{n}\,.
$$

\paragraph*{Squared error of cost.}
We now show how to augment the above proof technique
to yield the squared error bound in Theorem~\ref{thm:sample_complexity_cost}.
Note that from here on, we suppress
constants depending only on $\eta$
as explained in the notation section.
Begin by applying Young's inequality to
see
\begin{equation}\label{eqn:squared_error_bd}
    \E|S_n - S|^2 \leqslant 2\E|\Phi_n(f_n, g_n)
    - \Phi_n(f_\star, g_\star)|^2 + 
    2\E|\Phi_n(f_\star, g_\star) - \Phi(f_\star, g_\star)|^2\,.
\end{equation}
For the first term, 
observe that $\Phi_n(f_n, g_n) \geqslant \Phi_n(f_\star, g_\star)$. Since the square function is monotonic on the positive real line, we can apply Proposition~\ref{prop:emp_objective_PL} to yield
$$
\E|\Phi_n(f_n, g_n)
    - \Phi_n(f_\star, g_\star)|^2
\lesssim \E \|\nabla \Phi_n(f_\star, g_\star)\|^4_n\,.
$$
    
We observe that by Proposition~\ref{prop:bded_dual}
and~\eqref{eq:norm_gradphin}, $\|\nabla \Phi_n(f_\star, g_\star)\|_n \lesssim 1$ almost surely.
Thus, we can apply Lemma~\ref{lem:emp_dev_gradphi} to conclude
$$
\E|\Phi_n(f_n, g_n)
    - \Phi_n(f_\star, g_\star)|^2
\lesssim \E \|\nabla \Phi_n(f_\star, g_\star)\|^4_n\lesssim 
\E \|\nabla\Phi_n(f_\star, g_\star)\|_n^2 \lesssim \frac{1}{n}\,.
$$
For the second term in~\eqref{eqn:squared_error_bd}, we observe that
\begin{align*}
\E|\Phi_n(f_\star, g_\star) - \Phi(f_\star, g_\star)|^2
&\leqslant 4 \E|(\mu_n - \mu)(f_\star)|^2
+ 4 \E|(\nu_n - \nu)(g_\star)|^2 \\
&\qquad  \quad+\frac{4}{\eta^2}
\E|(\mu_n \otimes \nu_n - \mu \otimes \nu)(p_\star - 1)|^2 \\
&\lesssim \frac{1}{n} + \E|(\mu_n \otimes \nu_n - \mu \otimes \nu)(p_\star - 1)|^2\,,
\end{align*} where we use the uniform boundedness from
Proposition~\ref{prop:bded_dual}.
We observe that by the marginal constraints~\eqref{eqn:pop_marg}, Jensen's inequality, and~\eqref{eq:norm_gradphin},
$$
|(\mu_n \otimes \nu_n - \mu \otimes \nu)(p_\star - 1)|^2
= |(\mu_n \otimes \nu_n)(p_\star - 1)|^2
\leqslant \frac{1}{n} \sum_{i = 1}^n \Big( \frac{1}{n} \sum_{j = 1}^n
p_\star(X_i, Y_j) - 1\Big)^2 \leqslant \|\nabla \Phi_n(f_\star, g_\star)\|_n^2\,.
$$ Applying Lemma~\ref{lem:emp_dev_gradphi}
once more, we conclude that
$$
\E|(\mu_n \otimes \nu_n - \mu \otimes \nu)(p_\star - 1)|^2
\lesssim \frac{1}{n}\,.
$$ This yields the result.

\paragraph*{Map.}
We explain here how to prove Theorem~\ref{thm:sample_complexity_map}.
Recall that we use the notation
\begin{align*}
    b_\star(x) &= \E_{\pi_{\star}}[Y \, | \, X=x]=\int yp_\star(x, y) \ud \nu(y)\,,\\
 b_n(x) &= \E_{\pi_n}[Y \, | \, X=x]=\frac{1}{n} \sum_{j = 1}^n Y_j p_n(x, Y_j)\,.
\end{align*}
It holds
\begin{align}
\|b_\star - b_n \|^2_{L^2(\mu)} 
&\le 2 \Big\|b_\star - \frac{1}{n} \sum_{j = 1}^n
Y_j p_\star(\cdot , Y_j)\Big\|^2_{L^2(\mu)}
+ 2\Big\| \frac{1}{n} \sum_{j = 1}^n
Y_j p_\star(\cdot, Y_j) - b_n\Big\|^2_{L^2(\mu)} \nonumber \\
&\le 2\Big\|b_\star - \frac{1}{n} \sum_{j = 1}^n
Y_j p_\star(\cdot, Y_j)\Big\|^2_{L^2(\mu)}
+ \frac{1}{2n} \sum_{j = 1}^n\|p_n(\cdot,
Y_j)- p_\star(\cdot, Y_j)\|_{L^2(\mu)}^2\,,\label{eqn:pop_bary_bd}
\end{align}
where in the second inequality, we used Jensen's inequality together with the bound $\|Y_j\|\le 1/2$.
The first term is straightforward to control in
expectation:
\begin{align*}
\E \Big\|b_\star - \frac{1}{n} \sum_{j = 1}^n
Y_j p_\star(\cdot, Y_j)\Big\|^2_{L^2(\mu)} 
&= \frac{1}{n}\int\|yp_\star(x, y) - b_\star(x)\|^2 \ud \mu(x) \ud \nu (y)
\lesssim \frac{1}{n}\,,
\end{align*} where the last inequality follows by once again
applying
Proposition~\ref{prop:bded_dual}.
For the second term in~\eqref{eqn:pop_bary_bd},
observe that Proposition~\ref{prop:lipschitz} and Young's inequality implies
$$
\frac{1}{2n} \sum_{j = 1}^n\|p_n(\cdot,
Y_j)- p_\star(\cdot, Y_j)\|_{L^2(\mu)}^2
\lesssim \|\bar f_\star - f_n\|_{L^2(\mu)}^2 + \|\bar g_\star - g_n\|_{L^2(\nu_n)}^2\,.
$$
The next Lemma is useful in this
proof and the subsequent proof for the density.
It is immediate from Propositions~\ref{prop:bded_dual} and~\ref{prop:lipschitz}.
\begin{lemma}\label{lem:pop_to_emp_norms_expectation}We have
$$
\E\|f_n - \bar f_\star\|^2_{L^2(\mu)}
\lesssim \E\|g_n - \bar g_\star \|_{L^2(\nu_n)}^2 + \frac{1}{n}\,,
$$ and 
$$
\E\|g_n - \bar g_\star\|^2_{L^2(\nu)}
\lesssim \E\|f_n - \bar f_\star \|_{L^2(\mu_n)}^2 + \frac{1}{n}\,.
$$
\end{lemma}
Applying Lemma~\ref{lem:pop_to_emp_norms_expectation} yields
$$
\E \|b_\star - b_n \|^2_{L^2(\mu)} \lesssim
\E\|\bar g_\star - g_n\|^2_{L^2(\nu_n)} + \frac{1}{n}\,.
$$ 
To conclude, we combine~\eqref{eq:EB} together with Lemma~\ref{lem:emp_dev_gradphi} to get the desired bound on $\E\|\bar g_\star - g_n\|^2_{L^2(\nu_n)}$.

\paragraph*{Density.}

Proposition~\ref{prop:lipschitz} and Young's inequality implies
$$
\|p_n - p_\star\|_{L^2(\mu \otimes \nu)}^2\lesssim \|f_n-\bar f_\star\|^2_{L^2(\mu)}+\|g_n-\bar g_\star\|^2_{L^2(\nu)}\,.
$$
Hence by Lemma~\ref{lem:pop_to_emp_norms_expectation},
$$
\E\|p_n - p_\star\|_{L^2(\mu \otimes \nu)}^2\lesssim \E \big[
\|f_n - \bar f_\star \|_{L^2(\mu_n)}^2 + \|g_n - \bar g_\star \|_{L^2(\nu_n)}^2 \big]
+ \frac{1}{n}\,.
$$

 To conclude, we combine again~\eqref{eq:EB} together with Lemma~\ref{lem:emp_dev_gradphi}.

\paragraph*{Extending to the full results.}
To give the results in probability, we
replace each use of expectations with appropriate concentration
inequalities in the above proofs. Because the main ideas
are all present in the proofs in expectation, we leave these
extensions to Appendix~\ref{sec:appendix_sample_complexity}.

\subsection{Fluctuations of bounded test functions}\label{subsec:main_conv}

We now give the proof of our main result on
the difference $(\pi_n - \pi_\star)(\phi)$ where $\phi$ is
a bounded test function. Recall that
Theorem~\ref{thm:main_conv} says that, for all $t > 0$, with probability at least $1 - 8e^{-t^2}$,
$$
|(\pi_\star - \pi_n)(\phi)| \lesssim \frac{\|\phi\|_{L^{\infty}(\mu \otimes \nu)}\cdot t}{\sqrt{n}}\,.
$$
To begin, we assume
without loss of generality that
$\|\phi\|_{L^{\infty}(\mu \otimes \nu)} = 1$.
We break up the bound into two terms,
and use the uniform boundedness
of $\phi$ to yield
\begin{align*}
    |(\pi_n - \pi_\star)(\phi)|  &\leqslant
    \frac{1}{n^2} \sum_{i, j = 1}^n |p_n(X_i, Y_j) - p_\star(X_i, Y_j)|
     + \frac{1}{n^2} \Big| \sum_{i, j = 1}^n (p_\star(X_i, Y_j)\phi(X_i, Y_j) 
    - \pi_\star(\phi))\Big|\,.
\end{align*} The first term can be controlled using the results
from the previous section. Specifically, Theorem~\ref{thm:sample_complexity_density}
implies that with probability at least $1 - 16e^{-t^2}$,
$$
    \frac{1}{n^2} \sum_{i, j = 1}^n |p_n(X_i, Y_j) - p_\star(X_i, Y_j)|
  \lesssim \frac{t}{\sqrt{n}}\,.
$$
For the second term above, we can rewrite it as
$$
\frac{1}{n^2} \sum_{i, j = 1}^n (p_\star(X_i, Y_j)\phi(X_i, Y_j) 
    - \pi_\star(\phi)) = (\mu_n \otimes \nu_n - \mu \otimes \nu)(
    p_\star \phi)\,.
$$ Thus, this second term
resembles an empirical sample from $\mu \otimes \nu$,
except that we have joined together samples from $\mu$ and $\nu$.
This difference is, however, benign. Using a
trick from $U$-statistics~\cite{Hoe63},
we can, in fact, control the size
of the above process by that of a {\it bona fide} sample from $\mu \otimes \nu$
of size $n$.
This technique is encapsulated in the following Lemma.

\begin{lemma}\label{lem:u_statistics_trick}
Suppose $a \in L^{\infty}(\mu \otimes \nu)$ is such that
$(\mu \otimes \nu)(a) = 0$. Then, for all $t > 0$, with probability
at least $1 - 2e^{-t}$ over $\X$ and $\Y$,
$$
|(\mu_n \otimes \nu_n)(a)| \leqslant  \sqrt{\frac{2t}{n}} \cdot  \|a\|_{L^{\infty}(\mu \otimes \nu)} 
\,.
$$
\end{lemma}
\begin{proof}
Apply Chernoff's bound to see that for any $\lambda > 0$,
\begin{align*}
    \P_{\X, \Y}\big[ (\mu_n \otimes \nu_n)(a)> t \big]
    &\leqslant e^{-\lambda t}\E_{\X, \Y}\big[ \exp\big\{\lambda (\mu_n \otimes \nu_n)(a) \big\}\big]\,.
\end{align*}
We can now apply the key idea from Hoeffding's paper~\cite{Hoe63}.
By counting terms, we observe that we can rewrite $(\mu_n \otimes \nu_n)(a)$
as
$$
(\mu_n \otimes \nu_n)(a) = \frac{1}{n!}\sum_{\sigma \in \Sigma_n}
\frac{1}{n} \sum_{k = 1}^n a(X_k, Y_{\sigma(k)})\,,
$$ where $\Sigma_n$ is the set of permutations on $n$ elements.
Combining this observation with Jensen's inequality yields the bound
\begin{align*}
    \P_{\X, \Y}\big[  (\mu_n \otimes \nu_n)(a) > t \big]
    &\leqslant e^{-\lambda t}\E_{\X, \Y}\Big[ \exp\Big\{\frac{\lambda}{n!}
    \sum_{\sigma \in \Sigma_n}\frac{1}{n}
    \sum_{k = 1}^n a(X_k, Y_{\sigma(k)}) \Big\}\Big] \\
    &\leqslant e^{-\lambda t}
    \E_{\X, \Y} \Big[ \frac{1}{n!}\sum_{\sigma \in \Sigma_n}
    \exp\Big\{\frac{\lambda}{n}
    \sum_{k = 1}^n a(X_k, Y_{\sigma(k)}) \Big\}\Big]\,.
\end{align*} Now, we observe that for any fixed permutation $\sigma$,
the joint law of
$(X_1, Y_{\sigma(1)}), \ldots , (X_n, Y_{\sigma(n)})$,
is identical to that of $z_1, \ldots, z_n$ where $z_k \sim \mu \otimes \nu$
are independent and identically distributed. 
Let $\Z$ denote such an iid sample $(z_1, \ldots, z_n)$.
Thus we can change the order of summation to yield
\begin{align*}
    \E_{\X, \Y} \Big[ \frac{1}{n!}\sum_{\sigma \in \Sigma_n}
    \exp\Big\{\frac{\lambda}{n}
    \sum_{k = 1}^n a(X_k, Y_{\sigma(k)}) \Big\}\Big]
    &=\frac{1}{n!}\sum_{\sigma \in \Sigma_n}
     \E_{\X, \Y} \Big[ 
    \exp\Big\{\frac{\lambda}{n}
    \sum_{k = 1}^n a(X_k, Y_{\sigma(k)}) \Big\}\Big] \\
    &= \E_{\Z} \Big[ 
    \exp\Big\{\frac{\lambda}{n}
    \sum_{k = 1}^n a(z_k) \Big\}\Big]\,.
\end{align*}
Applying Hoeffding's Lemma and optimizing over $\lambda$
we find
$$
    \P_{\X, \Y}\big[(\mu_n \otimes \nu_n)(a)> t \big] 
    \leqslant \exp \Big\{-\frac{nt^2}{2\|a\|_{L^{\infty}(\mu \otimes \nu)}^2}
    \Big\}\, .
$$ The analogous argument works for the other tails. Combining these
bounds yields the result.
\end{proof}
With this Lemma~\ref{lem:u_statistics_trick} in hand, we can conclude the proof
of Theorem~\ref{thm:main_conv}
using the uniform boundedness of $p_\star$
from Proposition~\ref{prop:bded_dual}.

\subsection{Proofs for transfer learning}
\label{sec:pr_transfer}

\paragraph{Proof of Theorem~\ref{thm:sample_complexity_npls} for transfer regression.}
Observe
\begin{align*}
    \|h_\star - h_n\|_{L^2(\mu)}^2
    &\leqslant 2\big\|h_\star - \frac{1}{n} \sum_{j = 1}^n A_j p_\star(\cdot, Y_j)\big\|^2_{L^2(\mu)}
    + 2 \big\|\frac{1}{n} \sum_{j = 1}^n A_j (p_\star(\cdot, Y_j)
    - p_n(\cdot, Y_j)) \big\|_{L^2(\mu)}^2\\
    &\leqslant 2\big\|h_\star - \frac{1}{n} \sum_{j = 1}^n A_j p_\star(\cdot, Y_j)\big\|^2_{L^2(\mu)}
    + \frac{2}{n} \sum_{j = 1}^n \|p_\star(\cdot, Y_j)
    - p_n(\cdot, Y_j) \|^2_{L^2(\mu)}\,.
\end{align*} The result in expectation then
follows as in the proof of Theorem~\ref{thm:sample_complexity_map}
once we observe that
$$
 \E \big\|h_\star - \frac{1}{n} \sum_{j = 1}^n A_j p_\star(\cdot, Y_j)\big\|^2_{L^2(\mu)} = \frac{1}{n}
\E \|A_1 p_\star(\cdot, Y_1) - h_\star\|_{L^2(\mu)}^2
 \lesssim \frac{1}{n}\,.
 $$ 
 
\paragraph{Proof of Theorem~\ref{thm:classif} for transfer classification.}
We start essentially as in the proof of Theorem~\ref{thm:sample_complexity_npls} without integration. Applying the triangle inequality and Jensen's inequality, we get that for any $x$, it holds
\begin{align}\label{eqn:classification_main}
    |h_n(x)-h_\star(x)| \le \big|\frac1n \sum_{j=1}^n A_j p_\star(x, Y_j) - \E_\omega[Ap_\star(x,Y)]\big| + \frac1n\sum_{j=1}^n|p_\star(x,Y_j) -  p_n(x, Y_j)|\, .
\end{align}
To control the first term, we use the boundedness
from Proposition~\ref{prop:bded_dual} to apply Hoeffding's inequality to get that with probability $1-2e^{-t}$ it holds for any $x$ that
$$
\big|\frac1n \sum_{j=1}^n A_j p_\star(x, Y_j) - \E_\omega[Ap_\star(x,Y)]\big| \lesssim \sqrt{\frac{t}{n}}.
$$
For the second term, recall that~\eqref{eqn:pstar-pn} yields
$$
\frac{1}{n} \sum_{j = 1}^n |p_\star(x, Y_j) - p_n(x, Y_j)|
\lesssim |\bar f_\star(x) - f_n(x)| +\|\bar g_\star - g_n \|_{L^1(\nu_n)}.
$$ 
Moreover, controlling first term using Proposition~\ref{prop:lipschitz} yields
$$
\frac{1}{n} \sum_{j = 1}^n |p_\star(x, Y_j) - p_n(x, Y_j)|
\lesssim 
\big| \frac{1}{n} \sum_{j = 1}^ne^{-\eta\|x - Y_j\|^2 + \eta  g_\star(Y_j)}
- \nu( e^{-\eta\|x - \cdot\|^2 + \eta  g_\star(\cdot)})\big|  +\|\bar g_\star - g_n \|_{L^1(\nu_n)}\,.
$$
Using Hoeffding's inequality, we get that
with probability at least $1 - 2e^{-t}$,
$$
\big| \frac{1}{n} \sum_{j = 1}^ne^{-\eta\|x - Y_j\|^2 + \eta  g_\star(Y_j)}
- \nu( e^{-\eta\|x - \cdot\|^2 + \eta  g_\star(\cdot)})\big| 
\lesssim \sqrt{\frac{t}{n}} \, .
$$
Moreover, Equation~\eqref{eq:EB} together with Lemma~\ref{lem:eps_sub_g}
imply that with probability at least $1 - 4e^{-t}$,
$$
  \|\bar g_\star - g_n\|_{L^1(\mu_n)} \le \|\bar g_\star - g_n\|_{L^2(\mu_n)} \lesssim \sqrt{\frac{t}{n}} \, .
$$ 
Collecting terms, we get that for any $x$,
with probability
at least $1 - 8e^{-t}$,
\begin{equation}
    \label{eqn:devbdclassif}
    |h_n(x)-h_\star(x)| \lesssim \sqrt{\frac{t}{n}}\,.
\end{equation}
We are now in a position to apply Lemma~3.1 in~\cite{AudTsy07} which implies that under the Mammen-Tsybakov noise condition above, the pointwise deviation inequality~\eqref{eqn:devbdclassif} readily yields the desired bound on the excess-risk of the plug-in classifier.

\medskip
\paragraph*{Acknowledgements.} The authors gratefully acknowledge
the Simons Institute for the Theory of Computing semester program
``Geometric Methods in Optimization and Sampling", where part
of this work was carried out. 
Philippe Rigollet 
was supported by NSF awards IIS-1838071, DMS-2022448, and CCF-2106377.
Austin J. Stromme was partially supported by the Department of Defense (DoD) through the National Defense
Science \& Engineering Graduate Fellowship (NDSEG) Program.

\appendix 

\section{Omitted proofs for sample complexity}
\label{sec:appendix_sample_complexity}
In this section we give the proofs of
the high probability
statements in Theorems~\ref{thm:sample_complexity_cost},~\ref{thm:sample_complexity_map},
and~\ref{thm:sample_complexity_density}.
In section~\ref{subsec:key_tail}, we state and prove our main
tail bound.
In section~\ref{subsec:extended_control}
we state and prove the lemma comparing
the extended dual potentials to their non-extended versions.
The theorems are then proved in sections~\ref{subsec:proof_costs_full},
~\ref{subsec:proof_maps_full}, and~\ref{subsec:proof_densities_full},
respectively.

\subsection{The tail bound}\label{subsec:key_tail}
To provide results with control on the tails, we
frequently use the following direct consequence
of the bounded differences inequality.

\begin{lemma}\label{lem:bounded_differences}
Suppose $Z_1, \ldots, Z_m$ are independent mean-zero random variables
taking values in a Hilbert space $(H, \|\cdot\|_H)$.
Suppose there is some $C> 0$ such that for each $k = 1, \ldots, m$,
$\|Z_k\|_H \leqslant C$. Then, for all $t > 0$,
with probability at least $1 - 2e^{-t}$,
$$
\Big\| \frac{1}{m}\sum_{k = 1}^m Z_k \Big\|_H^2 \leqslant \frac{8C^2t}{m}\,.
$$
\end{lemma}

Using this concentration inequality, we can prove the following
Lemma which is at the heart of our results.

\begin{lemma}\label{lem:eps_sub_g} Define
$\nabla \Phi_n$ as in~\eqref{eq:gradphin}
and $\langle \cdot, \cdot\rangle_n$
and $\|\cdot\|_n$ as the
inner product and norm of $L^2(\mu_n) \times L^2(\nu_n)$, as in section~\ref{sec:preliminaries_ent_OT}.
Then for all $t > 0$, with probability at least
$1 - 4e^{-t}$ over both $\X$ and $\Y$,
$$
\|\nabla \Phi_n(f_\star, g_\star)\|^2_n\lesssim \frac{t}{n}\,.
$$
\end{lemma}
\begin{proof}
Recall from~\eqref{eq:norm_gradphin} that
\begin{equation}\label{eqn:norm_gradphin_star}
    \|\nabla \Phi_n (f_\star,g_\star)\|_{n}^2 =  \frac1n\sum_{i=1}^n \Big(1-\frac1n\sum_{j=1}^n p_\star(X_i,Y_j)\Big)^2+ \frac1n\sum_{j=1}^n \Big(1-\frac1n\sum_{i=1}^n p_\star(X_i,Y_j)\Big)^2\,,
\end{equation}
We show the bound for the first term, the proof for the second
term is analogous.

Let $A_j := (1 - p_\star(X_i, Y_j))_{i = 1}^n \in \R^n$ for $j = 1, \ldots, n$,
and put
$$
\bar{A}_n := \frac{1}{n} \sum_{j = 1}^n A_j\,.
$$
Conditionally on $\X$, the vectors $A_j$ are independent
and have zero mean by the marginal equation~\eqref{eqn:pop_marg}.
Moreover, using the uniform boundedness of $p_\star$
from Proposition~\ref{prop:bded_dual}, we find $ \|A_j\| \lesssim \sqrt{n}$.
Applying Lemma~\ref{lem:bounded_differences} we find
with probability at least $1- 2e^{-t}$ over $\Y$ with $\X$ fixed,
$$
\frac1n\sum_{i=1}^n \Big(1-\frac1n\sum_{j=1}^n p_\star(X_i,Y_j)\Big)^2 \lesssim \frac{t}{n}.
$$ Using the analogous statement for the second term
in~\eqref{eqn:norm_gradphin_star} yields the result.
\end{proof}

\subsection{Controlling the extended dual potentials}\label{subsec:extended_control}

We need the next Lemma, which
is the high probability analog to Lemma~\ref{lem:pop_to_emp_norms_expectation}.

\begin{lemma}\label{lem:extended_control}
Fix any $t > 0$, and let $(f_n, g_n)$ be extended as in section~\ref{subsec:canonical_extension}.
Then with probability at least $1 - 2e^{-t}$ over $\Y$
$$
\|\bar f_\star - f_n\|_{L^2(\mu)}^2\lesssim \|\bar g_\star - g_n\|_{L^2(\nu_n)}^2
+ \frac{t}{n}\,.
$$ Similarly, with probability at least $1 - 2e^{-t}$ over $\X$,
$$
\|\bar g_\star - g_n\|_{L^2(\nu)}^2\lesssim \|\bar f_\star - f_n\|_{L^2(\mu_n)}^2
+ \frac{t}{n}\,.
$$
\end{lemma}
\begin{proof}
We prove the first statement, the second is analogous. Note that it follows readily from Lemma~\ref{lem:bounded_differences} that
 with probability at least $1 - 2e^{-t}$ over $\Y$, it holds
$$
 \Big\| \frac{1}{n} \sum_{j = 1}^n e^{-\eta \|\cdot - Y_j\|^2 + \eta  g_\star(Y_j)} - 
   \nu(e^{-\eta \|\cdot - y\|^2 + \eta  g_\star(y)} ) \Big\|_{L^2(\mu)}^2 
  \lesssim \frac{t}{n}\,.
$$
Together with Proposition~\ref{prop:lipschitz}, this completes the proof.
\end{proof}

\subsection{Proof of tail control for costs}\label{subsec:proof_costs_full}

To yield the result, we can use the same proof
as in the expectation case but with tail control replacing
the expectation of $\Phi_n(f_\star, g_\star) - \Phi(f_\star, g_\star)$,
and $\|\nabla \Phi_n(f_\star, g_\star)\|_n$.
To begin, we write
\begin{equation}\label{eqn:emp_difference_triangle}
|S_n - S| \leqslant |\Phi_n(f_n, g_n) - \Phi_n(f_\star, g_\star)|
+ |\Phi_n(f_\star, g_\star) - \Phi(f_\star, g_\star)|\,.
\end{equation}
For the first term on the right-hand side above,
Proposition~\ref{prop:emp_objective_PL} implies
$$
|\Phi_n(f_n, g_n) - \Phi_n(f_\star, g_\star)|
= \Phi_n(f_n, g_n) - \Phi_n(f_\star, g_\star) \lesssim \|\nabla \Phi_n(f_\star, g_\star)\|^2_n\,.
$$ So using Lemma~\ref{lem:eps_sub_g} we find that
with probability at least $1 - 4e^{-t}$ over both $\X$ and $\Y$,
$$
|\Phi_n(f_n, g_n) - \Phi_n(f_\star, g_\star)| \lesssim \frac{t}{n}\,.
$$
For latter term in~\eqref{eqn:emp_difference_triangle}, we can write
$$
\Phi_n(f_\star, g_\star) - \Phi(f_\star, g_\star)
= (\mu_n \otimes \nu_n - \mu \otimes \nu)\big(f_\star  + g_\star - \frac{1}{\eta}
p_\star \big)\,.
$$ Using the uniform boundedness from Proposition~\ref{prop:bded_dual},
we can apply Lemma~\ref{lem:u_statistics_trick} to find that
with probability at least $1 - 2e^{-t}$ over $\X, \Y$
$$
|(\mu_n \otimes \nu_n - \mu \otimes \nu)\big(f_\star  + g_\star - \frac{1}{\eta}
p_\star \big)| \lesssim \sqrt{\frac{t}{n}}\,.
$$ 
Applying both of these tails bounds to~\eqref{eqn:emp_difference_triangle}
yields the result.

\subsection{Proof of tail control for maps}\label{subsec:proof_maps_full}

We start by bounding
$$
    \|b_\star -  b_n\|^2_{L^2(\mu)}
    \lesssim \Big\|\frac{1}{n} \sum_{j = 1}^n Y_jp_\star(\cdot, Y_j) - b_\star(\cdot)\Big\|^2_{L^2(\mu)}
    + \Big\| \frac{1}{n} \sum_{j = 1}^n Y_j(p_\star(\cdot, Y_j) - p_n(\cdot, Y_j))\Big\|_{L^2(\mu)}^2. 
$$ For the first term, we apply the bounded
differences inequality, Lemma~\ref{lem:bounded_differences}. In this way, we see that for all $t > 0$,
with probability at least $1 - 2e^{-t}$ over $\Y$,
$$
\Big\|\frac{1}{n} \sum_{j = 1}^nY_j p_\star(\cdot, Y_j) - b_\star(\cdot)\Big\|^2_{L^2(\mu)} \lesssim \frac{t}{n}\,.
$$
Hence, we can focus on the second term. 

Using Jensen's and Proposition~\ref{prop:lipschitz}
yields
\begin{align*}
\Big\| \frac{1}{n} \sum_{j = 1}^n Y_j(p_\star(\cdot, Y_j) -p_n(\cdot, Y_j))\Big\|^2_{L^2(\mu)} &\lesssim
\frac{1}{n} \sum_{j = 1}^n \big\|p_\star(\cdot, Y_j) -p_n(\cdot, Y_j)\big\|_{L^2(\mu)}^2 \\
&\lesssim \| \bar{f}_\star - f_n \|_{L^2(\mu)}^2 + \|\bar g_\star - g_n\|^2_{L^2(\nu_n)}\,.
\end{align*} 
Applying Lemma~\ref{lem:extended_control}
we obtain that with probability
at least $1 - 4e^{-t}$ over $\Y$,
$$
\|b_\star - b_n\|^2_{L^2(\mu)} \lesssim \frac{t}{n} + \|\bar g_\star - g_n\|^2_{L^2(\nu_n)}\,.
$$ Hence, we may use Equation~\eqref{eq:EB} to find that
$$
\|b_\star - b_n\|^2_{L^2(\mu)}  \lesssim \frac{t}{n}
+ \|\nabla \Phi_n(f_\star, g_\star)\|^2_n\,,
$$ and now we can use Lemma~\ref{lem:eps_sub_g} to conclude.

\subsection{Proof of tail control for densities}\label{subsec:proof_densities_full}

For the bound in $L^2(\mu_n \otimes \nu_n)$,
Proposition~\ref{prop:lipschitz} implies
$$
\|p_n - p_\star\|_{L^2(\mu_n \otimes \nu_n)}^2 \lesssim
\|\bar f_\star - f_n\|_{L^2(\mu_n)}^2 + \|\bar g_\star - g_n\|^2_{L^2(\nu_n)}\,.
$$ Applying~\eqref{eq:EB} and Lemma~\ref{lem:eps_sub_g},
we can conclude.

For the bound in $L^2(\mu \otimes \nu)$,
put
$\bar f_\star := f_\star + \nu_n(g_\star)$ and $\bar g_\star:=
g_\star - \nu_n(g_\star)$ are as in
Proposition~\ref{prop:lipschitz}.
Then Proposition~\ref{prop:lipschitz}
implies
$$
\|p_n - p_\star\|^2_{L^2(\mu \otimes \nu)} \lesssim 
\| \bar f_\star - f_n\|^2_{L^2(\mu)} + \| \bar g_\star - g_n\|^2_{L^2(\nu)}\,,
$$
Applying Lemma~\ref{lem:extended_control} we find that with
probability at least $1 - 4e^{-t}$ over $\X, \Y$
$$
\|p_n - p_\star\|^2_{L^2(\mu \otimes \nu)} \lesssim 
\|\bar f_\star - f_n\|_{L^2(\mu_n)}^2 + \|\bar g_\star - g_n\|^2_{L^2(\nu_n)}
+ \frac{t}{n}\,.
$$ Applying~\eqref{eq:EB} and Lemma~\ref{lem:eps_sub_g},
we can conclude.

\section{Further proofs}\label{sec:appendix_further}

\subsection{Proof of Proposition~\ref{prop:bded_dual}}\label{subsec:bded_dual}

The proof is essentially that of~\cite[Prop. 1]{mena2019statistical},
just with different notation and slightly different assumptions.
We include it in this appendix for the reader's convenience.
We begin by showing the result for $(f_\star, g_\star)$.

The $\mu$-marginal constraint for $\pi_\star$, namely~\eqref{eqn:pop_marg},
implies that for $\mu$-almost every $x$,
\begin{align*}
 1 &= \int e^{-\eta(\|x - y\|^2 - f_\star(x) - g_\star(y))}\ud\nu(y) \geqslant e^{-\eta (1 - f_\star(x))} \int e^{\eta g_\star(y)} \ud \nu(y) \geqslant  
 e^{-\eta (1 - f_\star(x))}\,,
 \end{align*} where the first inequality follows by Assumption~\ref{assump:bded_support},
 the second inequality by Jensen's together with the convention that $\nu(g_\star)=0$.  This implies
 $f_\star(x) \leqslant 1$ $\mu$-almost everywhere. Hence, we can
 use the marginal constraint~\eqref{eqn:pop_marg} for $\nu$-almost every $y$ to yield
 $$
    1 = \int e^{-\eta (\|x - y\|^2 - f_\star(x) - g_\star(y))} \ud \mu(x) \leqslant  e^{\eta (1 + g_\star(y))}\,.
$$ Whence $g_\star(y) \geqslant -1$ for $\nu$-almost every $y$.

We now claim that $\mu(f_\star) \geqslant 0$.
To see this, start by observing that the primal Sinkhorn problem~\eqref{eqn:entropic_OT_primal},
has a non-negative optimal value, and so by Theorem~\ref{thm:entropic_ot_main},
the dual objective evaluated at $(f_\star, g_\star)$ is non-negative. Hence
$$
    0 \leqslant \mu(f_\star) + \nu(g_\star) - \frac{1}{\eta}(\mu \otimes \nu)
    \Big( e^{-\eta(\|x - y\|^2 - f_\star(x) - g_\star(y))} \Big) + \frac{1}{\eta} \\
    = \mu(f_\star)\,,
$$ where we used the marginal constraint~\eqref{eqn:pop_marg} and the convention $\nu(g_\star)=0$. Using this fact, we can mimic the proof that
$f_\star(x) \leqslant 1$ above to show that $g_\star(y) \leqslant 1$ for $\nu$ almost all $y$.  We similarly find $f_\star(x) \geqslant -1$ for $\mu$ almost all $x$,
finishing the proof for $(f_\star, g_\star)$. 

The exact same argument shows $\|f_n\|_{L^{\infty}(\mu_n)}, \|g_n\|_{L^{\infty}(\nu_n)} \leqslant 1$. In fact, this readily yields a stronger result for canonical extensions. Indeed, recall that for every $x \in \R^d$, the canonical extension of $f_n$ defined in~\eqref{eqn:fn_canonical}
satisfies, for any $x \in \R^d$,
$$
f_n(x) = -\frac{1}{\eta} \ln \Big( \frac{1}{n} \sum_{j = 1}^n e^{-\eta\|x - Y_j\|^2 + \eta g_n(Y_j)} \Big) \geqslant
 -\frac{1}{\eta} \ln \Big( \frac{1}{n} \sum_{j = 1}^n e^{ \eta g_n(Y_j)} \Big) \geqslant -1\,.
$$ For $\mu$-almost every $x$, we also find
$$
f_n(x) = -\frac{1}{\eta} \ln \Big( \frac{1}{n} \sum_{j = 1}^n e^{-\eta\|x - Y_j\|^2 + \eta g_n(Y_j)} \Big) \leqslant
 -\frac{1}{\eta} \ln \Big( \frac{1}{n} \sum_{j = 1}^n e^{-\eta  + \eta g_n(Y_j)} \Big) \leqslant 2\,.
$$ Hence $\|f_n\|_{L^{\infty}(\mu)} \leqslant 2$.
Applying the same argument for $g_n$ completes the bounds on dual potentials.

Bounds on densities readily follow from the above bounds on dual potentials together with the definitions~\eqref{eqn:pop_scaling} and~\eqref{eqn:emp_decomp}.

\subsection{Proof of Proposition~\ref{prop:emp_objective_PL}}
\label{subsec:emp_objective_PL}

Proposition~\ref{prop:emp_objective_PL} follows from the following proposition.

\begin{prop}[Polyak-\L{}ojasiewicz inequality]\label{prop:sc_pl} 
Let $C \subset \cH$ be a  convex subset of a Hilbert space $\cH$. Let
$\rho \colon \cH \to \R$ be an
$\alpha$-strongly convex on $C$. Then, it holds
$$
\rho(v) - \inf_{\cH} \rho \leqslant\frac1{2\alpha} \|\nabla \rho(v) \|^2
\quad \quad \forall v \in C\,.
$$
\end{prop}
\begin{proof} By definition of strong convexity,
for all $v, w \in C$,
$$
\rho(w) \geqslant \rho(v) + \langle \nabla \rho(v), w - v \rangle + \frac{\alpha}{2}\|w - v\|^2\,.
$$ 
The minimum over $w \in \cH$ of the right-hand side is achieved at $w = v - \nabla \rho(v)/\alpha$ and is given by
$$
\rho(v) - \frac{1}{2\alpha}\|\nabla \rho(v)\|^2\,.
$$
Since $\rho(w)$ exceeds this quantity for all $w \in C$, we also have
$$
\inf_{\cH} \rho \geqslant \rho(v) - \frac{1}{2\alpha}\|\nabla \rho(v)\|^2\,.
$$ Re-arranging yields the result.
\end{proof}

\subsection{Proof of Proposition~\ref{prop:lipschitz}}
\label{subsec:lipschitz}

Define
$$
\tilde{f}_\star(x) :=  - \frac{1}{\eta} \ln \Big(\frac{1}{n}\sum_{j = 1}^n e^{-\eta\|x - Y_j\|^2 + \eta 
\bar g_\star(Y_j)}\Big)\,.
$$ 
We have, as in the proof of Proposition~\ref{prop:bded_dual},
the bound, $|\tilde{f}_\star(x)| \leqslant 3$. Moreover, we get
$$
 |\bar f_\star(x) - f_n(x)| \le |\bar f_\star(x) - \tilde{f}_\star(x)|
    + |\tilde{f}_\star(x) - f_n(x)|  \,.
$$
To bound the first term, we use~\eqref{eqn:fstar_from_gstar} together with the Lipschitzness
of the logarithm and exponential maps that follows from the uniform bounds on each term. We obtain the
following estimate
\begin{align*}
|\bar f_\star(x) - \tilde{f}_\star(x)|
    &= \Big| \frac{1}{\eta} \ln \Big(\frac{1}{n}\sum_{j = 1}^n e^{-\eta\|x - Y_j\|^2 + \eta 
\bar g_\star(Y_j)}\Big)-\frac{1}{\eta} \ln \Big(\int e^{-\eta\|x - y\|^2 + \eta 
\bar g_\star(y)}\ud \nu(y)\Big) \Big| \\
&= \Big| \frac{1}{\eta} \ln \Big(\frac{1}{n}\sum_{j = 1}^n e^{-\eta\|x - Y_j\|^2 + \eta 
 g_\star(Y_j)}\Big)-\frac{1}{\eta} \ln \Big(\int e^{-\eta\|x - y\|^2 + \eta 
 g_\star(y)}\ud \nu(y)\Big) \Big| \\
    &\lesssim \Big| \frac{1}{n} \sum_{j = 1}^n e^{-\eta \|x - Y_j\|^2 + \eta   g_\star(Y_j)} - 
   \nu(e^{-\eta \|x - \cdot\|^2 + \eta   g_\star(\cdot)} ) \Big| \,.
\end{align*}
Using the same technique we get the following bound for the second term:
\begin{align*}
|\tilde{f}_\star(x) - f_n(x)|
    &= \Big| \frac{1}{\eta} \ln \Big(\frac{1}{n}\sum_{j = 1}^n e^{-\eta\|x - Y_j\|^2 + \eta 
g_n(Y_j)}\Big)-\frac{1}{\eta} \ln \Big(\frac{1}{n}\sum_{j = 1}^n e^{-\eta\|x - Y_j\|^2 + \eta 
\bar g_\star(Y_j)}\Big) \Big| \\
    &\lesssim \|\bar g_\star- g_n\|_{L^1(\nu_n)} \,.
\end{align*}
This yields the bound on $|\bar f_\star(x) - f_n(x)|$ and the bound of $|\bar g_\star(y) - g_n(y)|$ follows using the same argument.

To prove the bound on $|p_\star(x,y)-p_n(x,y)|$, we again apply the boundedness
from Proposition~\ref{prop:bded_dual} and the Lipschitzness of
the exponential on bounded intervals to yield
\begin{align*}
    |p_n(x, y) - p_\star(x, y)|
    &= \big| e^{-\eta\|x - y\|^2 + \eta f_n(x) + \eta g_n(y)}
    - e^{-\eta \|x -y\|^2 + \eta f_\star(x) + \eta g_\star(y)} \big| \\
    &= \big| e^{-\eta\|x - y\|^2 + \eta f_n(x) + \eta g_n(y)}
    - e^{-\eta \|x -y\|^2 + \eta \bar f_\star(x) + \eta \bar g_\star(y)} \big| \\
    &\lesssim |f_n(x) + g_n(y) - \bar f_\star(x)
    - \bar g_\star(y)| \\
    &\le |f_n(x)  - \bar f_\star(x)|  +
    |g_n(y) 
    - \bar g_\star(y)|\,.
\end{align*}

\bibliographystyle{alpha}
\bibliography{annot}
\end{document}